\date{12 May 2008} 
\title{Growth of the Number of Spanning Trees\\ of the Erd\H{o}s-R\'enyi Giant Component}
\author{Russell Lyons
\thanks{Partially supported by NSF grants DMS-0406017 and DMS-0705518 and
Microsoft Research.}
\thanks{
Department of Mathematics,
Indiana University,
Bloomington, IN 47405-5701. Email: \nobreak rdlyons@indiana.edu.}
\and Ron Peled \thanks{Supported by Microsoft Research and NSF grant DMS-0605166.}
\thanks{Department of Statistics, UC Berkeley. Email: \nobreak peledron@stat.berkeley.edu.} \and Oded Schramm\thanks{Microsoft Research.}}

\documentclass[12pt,naturalnames]{article}
\usepackage{amsmath}
\usepackage{amssymb}
\usepackage{amsthm}
\usepackage{amsfonts}
\usepackage{graphicx}

\def\thmenv#1#2#3{\begin{#1} \label{#1:#2} #3 \end{#1}}

\def\procl#1.#2 #3\endprocl{%
       \ifx#1t\thmenv{theorem}{#2}{#3}\fi
       \ifx#1l\thmenv{lemma}{#2}{#3}\fi
       \ifx#1p\thmenv{proposition}{#2}{#3}\fi
       \ifx#1c\thmenv{corollary}{#2}{#3}\fi
       \ifx#1d\thmenv{definition}{#2}{#3}\fi
       \ifx#1g\thmenv{conjecture}{#2}{#3}\fi
       \ifx#1q\thmenv{question}{#2}{#3}\fi
       \ifx#1r\thmenv{remark}{#2}{{\rm #3}}\fi
    }%

\def\rref#1.#2/{%
      \ifx #1sSection~\ref{sect:#2}\fi
      \ifx #1tTheorem~\ref{theorem:#2}\fi
      \ifx #1lLemma~\ref{lemma:#2}\fi
      \ifx #1cCorollary~\ref{corollary:#2}\fi
      \ifx #1pProposition~\ref{proposition:#2}\fi
      \ifx #1dDefinition~\ref{definition:#2}\fi
      \ifx #1gConjecture~\ref{conjecture:#2}\fi
      \ifx #1qQuestion~\ref{question:#2}\fi
      \ifx #1e(\ref{eq:#2})\fi
      \ifx #1b(\cite{#2})\fi
        }

\def\rlabel e.#1 #2{\begin{equation} \label{eq:#1} #2 \end{equation}}

\def\proof{\medbreak\noindent{\it Proof.\enspace}}

\def\proofof #1.#2 {\medbreak\noindent
     {\it Proof of \rref #1.#2/.}\enspace}

\newif\iffigures\figurestrue
\figuresfalse

\newif\ifhyper\IfFileExists{hyperref.sty}{\hypertrue}{\hyperfalse}

\ifhyper\usepackage{hyperref}
\def\hitem#1#2{\item[\hypertarget{#1}{#2}]\expandafter\gdef\csname LBL#1ITM\endcsname{#2}}
\def\iref#1{\hyperlink{#1}{\csname LBL#1ITM\endcsname}}
\else
\def\hitem#1#2{\item[{#2}]\expandafter\gdef\csname LBL#1ITM\endcsname{#2}}
\def\iref#1{{\csname LBL#1ITM\endcsname}}
\fi

\newif\ifdraft
\drafttrue
\long\def\note#1/{\ifdraft {\bf [#1]}\fi}
\long\def\comment#1{}
\long\def\old#1{}

\numberwithin{equation}{section}
\numberwithin{figure}{section}

\newtheorem{theorem}{Theorem}
\numberwithin{theorem}{section}
\newtheorem{corollary}[theorem]{Corollary}
\newtheorem{lemma}[theorem]{Lemma}

\theoremstyle{remark}
\theoremstyle{remark}\newtheorem{remark}[theorem]{Remark}

\let\qqed=\qed
\def\QED{\qqed\medskip}
\let\qed=\QED

\newcommand{\R}{\mathbb{R}}

\newcommand{\Z}{\mathbb{Z}}
\newcommand{\N}{\mathbb{N}}

\def\Im{{\rm Im}\,}
\def\Re{{\rm Re}\,}
\def\SLEkk#1/{$\mathrm{SLE}(#1)$}
\def\SLEr#1/{$\mathrm{SLE(\kappa;#1)}$}
\def\SLEkr#1;#2/{$\mathrm{SLE(#1;#2)}$}
\def\SLEk/{\SLEkk{\kappa}/}
\def\SLEtwo/{\SLEkk2/}
\def\SLE/{$\mathrm{SLE}$}
\def\SLEab/{\SLEkr 4; {a/\hco-1}, {b/\hco-1}/}

\def\Ito/{It\^o}

\def \P {{\bf P}}
\def\md{\mid}
\def\Bb#1#2{{\def\md{\bigm| }#1\bigl[#2\bigr]}}

\def\Pb{\Bb\P}
\def\Eb{\Bb\E}

\def \E {{\bf E}}

\def\les{\le_1}
\def\sles{\les^{\mathcal L}}

\def \proof {{ \medbreak \noindent {\bf Proof.} }}
\def\proofof#1{{ \medbreak \noindent {\bf Proof of #1.} }}

\def\FF{{\rm F}}
\def\II{{\rm I}}

\def\bl{\bigl}\def\br{\bigr}\def\Bl{\Bigl}\def\Br{\Bigr}

\def\LWbuild{MR2060629}
\def\Lest{MR2160416}
\def\PitmanEnum{MR1630413}
\def\Bollobas{MR2002j:05132}

\def\AldousPitman{MR1641670}
\def\Piau98{MR1634413}
\def\Ltrent{Lyons:trent}
\def\BSpyond{MR97j:60179}
\def\ErdRen{MR0125031}
\def\HoJa{MR2035627}

\def\GW{\PGW}
\def\cp{\tau}   
\def\G{{\cal G}}
\def\PGW{\textsf{PGW}}
\def\rtd{\rho}  
\def\gh{G}  
\def\bp{o}
\def\CO#1{[#1)}        
\def\verts{\textsf{V}}

\newcommand{\nchoosek}[2]{\begin{pmatrix}#1\\#2\end{pmatrix}}
\def\G{{\mathcal G}}
\def\pc{p_{\rm c}}

\def\noopsort#1{}

\newcommand{\W}{\mathcal{W}}

\begin{document}
\maketitle

\begin{abstract}
The number of spanning trees in the giant component of the random graph
$\G(n, c/n)$ ($c>1$) grows like $\exp\big\{m\big(f(c)+o(1)\big)\big\}$ as $n\to\infty$,
where $m$ is the number of vertices in the giant component.
The function $f$ is not known explicitly, but we show that it is
strictly increasing and infinitely differentiable.
Moreover, we give an explicit lower bound on $f'(c)$.
A key lemma is the following.
Let $\GW(\lambda)$ denote a Galton-Watson tree having Poisson
offspring distribution with parameter $\lambda$.
Suppose that $\lambda^*>\lambda>1$. We show that
$\GW(\lambda^*)$ conditioned to survive forever stochastically
dominates $\GW(\lambda)$ conditioned to survive forever.
\end{abstract}

\section{Introduction}

Methods of enumeration of spanning trees in a finite graph $\gh$ and relations
to various areas of mathematics and physics have been investigated for more
than 150 years.
The number of spanning trees is often called the {\bf complexity} of the
graph, denoted here by $\cp(\gh)$.
The usual Erd\H{o}s-R\'enyi
model of random graphs, $\G(n, p)$, is a graph on $n$ vertices,
each pair of which is connected by an edge with probability $p$,
independently of other edges.
Fix $c > 1$.
It is well known that with probability approaching 1 as $n \to\infty$,
the largest component of $\G(n,c/n)$ has size proportional to $n$,
while the second largest component is of logarithmic size.
(See, e.g., \cite{\ErdRen} or \cite{\Bollobas}.)
The largest component is thus called the {\bf giant component}
and will be denoted by $\gh_n=\gh_n(c)$.
As an example of a general theory, \cite{\Lest} proved that there is a
number $f(c)$ such that
$$
f(c)
=
\lim_{n \to\infty} \frac{1} {|\verts(\gh_n)|}  \log \cp(\gh_n)
$$
in probability.
In the same paper it was shown that $f(c) > 0$ for $c > 1$, that $f(1^+) = 0$, and
that $f$ is continuous on $\CO{1,\infty}$.
\cite{\Lest} asked whether $f$ is strictly increasing 
and real analytic on $(1,\infty)$.
Note that as $c$ increases, both the number of trees $\cp(\gh_n)$ as well
as the number of vertices $|\verts(\gh_n)|$
increase, so that it is not clear which increase dominates.
Here we prove that $f$ is strictly increasing and $C^\infty$; prior to our
work, it was not known even that $f$ was non-decreasing.

Let $\PGW(c)$ be
the law of a rooted Galton-Watson tree $(T, o)$
with Poisson($c$) offspring distribution.
Write $\PGW^*(c)$ for the law of
$\PGW(c)$ conditioned on non-extinction.
Sometimes we also write this measure as $\PGW^*_c$.
The event of extinction has probability $q(c)$, which is well known to be
the smallest positive solution of the equation
\rlabel e.extinct
{q(c) = e^{-c (1-q(c))}
\,.
}
Let $p_k(x;\gh)$ denote the probability that simple random walk on a graph
$\gh$ started at a vertex $x$ is back at $x$ after $k$ steps.
\cite{\Lest} proved that
\rlabel e.treeent
{f(c)
=
\int \Big(\log \deg_T(\bp) - \sum_{k \ge 1} \frac{1 }{ k} p_k(\bp;T)\Big)
\,d\PGW^*_c(T, \bp)
\,.
}

\procl t.incrH
The function $f$ is strictly increasing and $C^\infty$ on $(1, \infty)$.
In fact,
$$
f'(c) >
\frac{(c-1)e^{-c q(c) }}{c^2}
> 0
$$
for $c > 1$.
\endprocl

From \rref e.treeent/,
it is not hard to see that for any $\epsilon > 0$, we have that 
$\big(f(c+ \epsilon) - f(c)\big)/\epsilon \sim 1/c$ as $c \to\infty$.
Since 
\rlabel e.duality
{c e^{-c} = c q(c) e^{-c q(c)}}
and the function $x \mapsto x e^{-x}$
is unimodal in $(0, \infty)$ and vanishes at 0 and $\infty$,
it follows that $\lim_{c \to\infty}
c q(c) = 0$. Using this, we find that
our lower bound for $f'(c)$ in \rref t.incrH/ has the same asymptotic,
$1/c$, as $c \to\infty$.
We do not have any information on $f'(1)$.

A key lemma to prove \rref t.incrH/ is the following:

\begin{theorem}\label{t.main}
If $c'>c\ge 1$, then \emph{$\GW^*(c')$} stochastically dominates
\emph{$\GW^*(c)$}.
\end{theorem}

Here, $\GW^*(1)$ denotes the weak limit as $c \downarrow 1$ of $\GW^*(c)$;
see \cite{\AldousPitman}, Lemma 23.

We now recall what the stochastic domination referred to in the theorem means.
If $(T,o)$ and $(T',o')$ are rooted trees, we say that $(T,o)$ dominates
$(T',o')$ if there is an isomorphism from $T'$ to a subtree of $T$ that takes
$o'$ to $o$.
A probability measure on the collection of rooted trees is said to
stochastically dominate another probability
measure on the collection of rooted trees if they may be coupled so
that the sample from the first measure a.s.\ dominates the sample from
the second measure.

Of course, $\GW(c')$ dominates $\GW(c)$ when $c'>c$. It is the conditioning
that makes Theorem~\ref{t.main} nontrivial. Indeed, the offspring
distribution that has 1 or 3 children with probability 1/2 each
stochastically dominates the offspring distribution that has 0 or 3
children with probability 1/2 each, but if we condition on survival, the
domination does not persist since conditioning does not change the former,
but forces the latter to have 3 children of the root.

\section{Tree Domination} \label{tree_domination_section}

Let $T_n=T_n(\lambda)$ be a $\GW(\lambda)$ tree conditioned to have $n$
vertices, where $n\in\N_+\cup\{\infty\}$.
We consider the values of $T_n$ to be equivalence classes of rooted trees
under isomorphisms that preserve the root.
It is easy to check that the distribution of $T_n$ does not depend on
$\lambda$.
It turns out that it is the same as the distribution obtained by forgetting
the labels of a uniform tree on $n$ vertices with uniform root.
Also, the probability that a $\GW(\lambda)$ tree has $k$ vertices is given
by the Borel($\lambda$) distribution, namely,
\rlabel e.Borel
{\frac{(\lambda\,e^{-\lambda})^k\,k^{k-1}} {\lambda k!}\,.}
These facts are well known and have a variety of proofs; see
\cite{\PitmanEnum} for some of them.

\cite{\LWbuild} show the following:

\begin{theorem}\label{t.fin}
$T_{n+1}$ stochastically dominates $T_n$ for every $n \in \N_+$.
\end{theorem}

More precisely, in their Theorem 4.1, for each $d \ge 2$
they show such a statement for conditioned trees having offspring
distribution binomial with parameters $(d, 1/d)$.
Taking a limit as $d \to\infty$ gives Theorem \ref{t.fin}.
It is interesting to note that this is the same as saying that a uniformly
rooted uniform tree on $n+1$ vertices dominates a uniformly rooted uniform
tree on $n$ vertices.

Define $\theta(c)$ as the survival probability of $\GW(c)$, that is,
$$
\theta(c):=1-q(c)\,.
$$
For $\mu>\lambda>0$, set
$$
\alpha(\lambda,\mu):=
\log\frac{e^\mu-1}{\mu}-
\log\frac{e^\lambda-1}{\lambda}>0\,.
$$

\procl l.ders
For $\mu>\lambda>1$, we have
$$
\alpha(\lambda,\mu) < \mu-\lambda
$$
and
$$
\alpha\big(\lambda \theta(\lambda), \mu \theta(\mu)\big)
>
\lambda q(\lambda)- \mu q(\mu)
\,.
$$
\endprocl

\proof
The first inequality states that $x \mapsto \log\big((e^x - 1)/x\big) - x$
is decreasing for $1 < x < \infty$, which in turn is a consequence of the inequality $e^x > 1 + x$.

By \rref e.extinct/, we have
$$
\frac{e^{\lambda \theta(\lambda)}-1}{\lambda \theta(\lambda)}
=
\frac{1 - q(\lambda)}{\lambda q(\lambda) \theta(\lambda)}
=
\frac{1}{\lambda q(\lambda)}
\,.
$$
Therefore 
\rlabel e.simpler
{\alpha\big(\lambda \theta(\lambda), \mu \theta(\mu)\big)
=
\log \big(\lambda q(\lambda)\big) - \log \big(\mu q(\mu)\big)
\,.}
We next note that $x \mapsto xe^{-x}$ is strictly increasing on $(0, 1)$ and strictly decreasing on
$(1,\infty)$. Recalling from \rref e.duality/ that
$\lambda\exp(-\lambda)=\lambda q(\lambda)\exp(-\lambda q(\lambda))$ and
using that $\lambda>1$, we deduce that $\lambda q(\lambda)<1$. We similarly deduce that $\mu q(\mu)<\lambda q(\lambda)$ since $\mu>\lambda$. The second claimed inequality therefore follows from \rref e.simpler/ and
the fact that $x \mapsto \log x - x$ is increasing on $(0, 1)$.
\QED


Let $Q_\lambda$ denote a Poisson$(\lambda)$ random variable
and $Q^*_\lambda$ denote
a random variable whose distribution is the same
as that of $Q_\lambda$ conditioned on $Q_\lambda>0$.

\begin{lemma}\label{l.decomp}
Let $\mu>\lambda>0$ and set
$
\alpha=\alpha(\lambda,\mu)
$.
Then $Q^*_\mu$ stochastically dominates the sum of mutually independent
copies of $Q^*_\lambda$ and $Q_\alpha$.
Moreover, this does not hold for any larger $\alpha$.
\end{lemma}

\proof
Consider some $\beta\in (0,\mu-\lambda)$, and
let $Z$ denote the sum of two mutually independent copies of
$Q^*_\lambda$ and $Q_\beta$. For $k\in\N_+$, set
\rlabel e.sumdist
{a_k  := \Pb{Z=k}
= \frac{e^{-\lambda-\beta}}{1-e^{-\lambda}}
\sum_{j=1}^k \frac{\lambda^j\,\beta^{k-j}}{j!\,(k-j)!}
= \frac{e^{-\beta}}{e^{\lambda}-1}\,\frac{(\lambda+\beta)^k-\beta^k}{k!}
\,,}
and
$$
b_k  := \Pb{Q^*_\mu=k} = \frac{1}{e^\mu-1}\,\frac{\mu^k}{k!}\,.
$$
In order for $Q^*_\mu$ to dominate $Z$, it is necessary that $a_1\ge b_1$.
This translates and simplifies to $\beta\le \alpha$.
Now fix $\beta=\alpha$; note that $\alpha < \mu - \lambda$ by \rref l.ders/.
Let $J$ be the set of $k\in\N_+$ such that
$a_k\ge b_k$. We claim that there is a $k_0\in\R$ such that
$J=\N_+\cap [1,k_0]$. Before proving the claim, we shall demonstrate that the
lemma follows from it. Indeed, to prove domination it is sufficient
to show that
\begin{equation}
\label{e.tail}
\forall k\in\N_+\qquad \sum_{j=k}^\infty b_j\ge \sum_{j=k}^\infty a_j\,.
\end{equation}
This is clear if $k>k_0$, because $b_j>a_j$ for $j>k_0$.
On the other hand, if $k\le k_0$, then $a_j\ge b_j$ for $j\le k$,
which gives
$$
\sum_{j=1}^k b_j\le \sum_{j=1}^k a_j\,.
$$
Now subtracting both sides from $\sum_{j\in\N_+} b_j=1=\sum_{j\in\N_+} a_j$
gives~\eqref{e.tail}.

It remains to prove that $J=\N_+\cap[1,k_0]$ for some $k_0\in\R$.
Note that
$$
\frac{a_k}{b_k}=
\frac{(e^\mu-1)\,e^{-\beta}}
{e^{\lambda}-1}
\,\left(\Bl(\frac{\lambda+\beta}\mu\Br)^k-\Bl(\frac{\beta}{\mu}\Br)^k\right).
$$
Now think of the right-hand side as a function $g(k)$ of positive real $k$.
As such, it may be written in the form $A\,(B^k-C^k)$, with constants
$A,B,C$ satisfying $A>0$
and $1>B>C>0$. We claim that $g$ does not have any local minimum.
Indeed, $g'(k)=A\,(B^k\log B-C^k\log C)$ and
$g''(k)= A\,\bl(B^k(\log B)^2-C^k(\log C)^2\br)$. Therefore,
when $g'(k)=0$, we have $(\log B)/(\log C)=C^k\,B^{-k}$
and so $g''(k)= A\,(\log C)^2\, C^k\,(C^{k}\,B^{-k}-1)<0$.
This verifies that $g$ does not have a local minimum. Hence, the set of
$k\in (0,\infty)$ such that $g(k)\ge 1$ is an interval. By our choice of
$\beta=\alpha$, this interval contains $1$. This proves the claim, and
completes the proof of the lemma.
\QED

Given a rooted tree $T$ with root $\bp$ and a node
$v$ in $T$, let $N(v)$ denote the cardinality of the set
of nodes in the subtree of $T$ corresponding to $v$, that is,
the number of nodes in $T$ that are not in the connected component
of $\bp$ in $T\setminus \{v\}$.
Given a random rooted tree $T$, let $n_k=n_k(T)$ denote
the number of children $v$ of the root satisfying $N(v)=k$.

Let $T(\lambda)$ denote a sample from $\GW(\lambda)$.
Note that
the random variables $\bl(n_k(T(\lambda)):k\in\N_+\cup\{\infty\}\br)$
are independent Poisson random variables.
It follows that the random variables
$ \bl(n_k(T_\infty(\lambda)):k\in\N_+\cup\{\infty\}\br)$
are independent, and $n_k(T_\infty(\lambda))$
has the same law as $n_k(T(\lambda))$ when $k\in \N_+$,
while $n_\infty(T_\infty(\lambda))$ has the law of
$n_\infty(T(\lambda))$ conditioned on being positive.

By \rref e.Borel/, we have
$$
\forall k\in\N_+\qquad
\Eb{n_k(T(\lambda))}= \frac{(\lambda\,e^{-\lambda})^k\,k^{k-1}} {k!}\,.
$$
Observe that this is monotone decreasing in $\lambda$
in the range $\lambda\ge 1$.

Since $n_\infty(T(\lambda))$ is Poisson with parameter
$\lambda\,\theta$, we have
$$
\sum_{k\in \N_+} \Eb{n_k(T(\lambda))}= (1-\theta)\,\lambda\,.
$$

\medskip
If $T$ and $T'$ are rooted trees, we write
$T\les T'$ if there is an injective map
$i$ from the children of the root in $T$
to the children of the root in $T'$
such that $N(i(v))\ge N(v)$
for every child $v$ of the root in $T$.
If $T$ and $T'$ are random rooted trees,
we write $T\sles T'$ if
$T$ and $T'$ may be coupled in such a way that
 $T\les T'$ a.s.
Observe that $\sles$ is a partial order relation.

Theorem~\ref{t.main} will follow easily from the following lemma.

\begin{lemma}\label{l.les}
Let $\mu>\lambda>1$ and let $n\in\N_+$.
Then $T_n\sles T_{n+1}\sles T_\infty(\lambda)\sles T_\infty(\mu)$.
\end{lemma}

\proof
We start by proving $ T_\infty(\lambda)\sles T_\infty(\mu)$.
Let $(Z_k:k\in\N_+)$ and $(Z_k':k\in\N_+)$ be independent Poisson random
variables with $\Eb{Z_k}=\Eb{n_k(T(\mu))}$
and $\Eb{Z_k'}=\Eb{n_k(T(\lambda))}-\Eb{n_k(T(\mu))}$.
(Recall that the latter is non-negative.)
Let $Z':=\sum_{k\in\N_+} Z_k'$, which is a Poisson random variable.
By the above,
$$
\Eb{Z'}
=
\bl(1-\theta(\lambda)\br)\lambda-
\bl(1-\theta(\mu)\br)\mu\,.
$$
By \rref l.ders/, we have
\begin{equation}
\alpha\big(\lambda \theta(\lambda), \mu \theta(\mu)\big)
\ge 
\lambda q(\lambda)- \mu q(\mu)
=\Eb{Z'}.
\end{equation}
Consequently, by Lemma~\ref{l.decomp}, $n_\infty(T_\infty(\mu))$
may be coupled to dominate  $n_\infty(T_\infty(\lambda))$
plus an independent copy of $Z'$. Thus, we may take
$n_\infty(T_\infty(\lambda))$ independent from
$(Z_k:k\in\N_+)$ and $(Z'_k:k\in\N_+)$ and take $n_\infty(T_\infty(\mu))$
independent from $(Z_k:k\in\N_+)$ such that
$n_\infty(T_\infty(\mu))\ge Z'+n_\infty(T_\infty(\lambda))$.
For $k\in\N_+$ we take $n_k(T_\infty(\lambda))=Z_k+Z'_k$,
and $n_k(T_\infty(\mu))=Z_k$.
Since $Z'=\sum_{k\in\N_+} Z'_k$, with these choices we have
$T_\infty(\lambda)\les T_\infty(\mu)$.
This proves that $T_\infty(\lambda)\sles T_\infty(\mu)$.

The fact that $T_n\sles T_{n+1}$ follows from Theorem~\ref{t.fin}.
Recall that the limit in law of $T_n$ as $n\to\infty$
is the same as the limit in law of $T_\infty(\lambda)$ as
$\lambda\searrow 1$; see \cite{\AldousPitman}, Lemma 23.
Let $T_\infty(1)$ denote a random tree with this
limit law. Then $T_n\sles T_\infty(1)$.
By taking the limit as $\lambda'\searrow 1$ in
$T_\infty(\lambda')\sles T_\infty(\lambda)$,
we find that $T_\infty(1)\sles T_\infty(\lambda)$.
Thus, $T_n\sles T_\infty(\lambda)$ follows.
\QED

\proofof{Theorem~\ref{t.main}}
This follows by repeatedly applying Lemma~\ref{l.les}
at each node of the $T_\infty(\lambda)$ tree, while keeping the
corresponding couplings appropriately conditionally independent.
\QED

\section{Return Probabilities}

A general result on monotonicity \cite{\Ltrent}, combined with Theorem \ref
{t.main} implies
the monotonicity claim in \rref t.incrH/.
Here, we analyze in more detail the expression \rref e.treeent/
in order to gain an explicit lower bound on the derivative of $f(c)$, which
that general result does not supply.

Our main aim in this section is to prove the following result:

\procl t.return
The expression
\emph{$$
\int \sum_{k \ge 1} \frac{1 }{ k} p_k(\bp;T) \,d\PGW^*_c(T, \bp)
$$}
is monotonic decreasing in $c > 1$.
\endprocl

In light of \rref e.treeent/ and Theorem \ref{t.main}, 
this implies the monotonicity claim in \rref t.incrH/ 
and will lead to an explicit lower bound on the derivative in Section
\ref{deriv}.
It also implies the following lower bound for $f(c)$ itself: 
$$
f(c)
\ge
\sum_{k \ge 1} \frac{e^{-c} c^k (1 - q(c)^k) \log k}{\theta(c) k!}
-
\sum_{k \ge 0} \frac{e^{-1} \log (1+k)}{k!}
\ge
0
\,.
$$
To see this, note first that by, say, \rref e.sumdist/, we have that 
$$
\PGW^*_c[\deg_T(\bp) = k]
=
\frac{e^{-c} c^k (1 - q(c)^k) }{\theta(c) k!}
\,.
$$
Second, recall that $\lim_{c \downarrow 1} f(c) = 0$.
Therefore, \rref t.return/ and \rref e.treeent/ imply that
$$
f(c) \ge
\int \log \deg_T(\bp) \,d\PGW^*_c(T, \bp)
-
\int \log \deg_T(\bp) \,d\PGW^*_1(T, \bp)
\ge
0
\,,
$$
and this equals the above expression by the well-known form of $\PGW^*(1)$
(\cite{\AldousPitman}, Corollary 3).
This lower bound should be compared to the trivial upper bound
$$
f(c)
\le
\sum_{k \ge 1} \frac{e^{-c} c^k (1 - q(c)^k) \log k}{\theta(c) k!}
\,.
$$

To prove \rref t.return/,
let $V(s, T, \bp) := \sum_{k \ge 0} p_k(\bp;T) s^k$.
Since
$$
\int_0^1 \frac{\E[V(s, T, \bp)] - 1}{s} \,d s
=
\int \sum_{k \ge 1} \frac{1 }{ k} p_k(\bp;T) \,d\PGW^*_c(T, \bp)
\,,
$$
\rref t.return/ will be a consequence of the following result:

\procl t.rets
For all $s\in(0,1)$,
the expectation \emph{$\int V(s, T, \bp)\,d\PGW^*_c(T, \bp)$} is 
decreasing in $c > 1$.
\endprocl

Fix $\mu>\lambda>1$ and let $T$ and $T'$ have the distributions
$\PGW^*(\lambda)$ and $\PGW^*(\mu)$, respectively.
Let $X$ count the number of visits to the root in a random walk
on the tree $T$ started from the root in which at each step the
walker has probability $1-s$ to die, independent of the other steps
(note that $X\ge1$ since we start from the root). Let $X'$ be
the same for a walk on $T'$.
Because $V(s, T, \bp) = \E[X]$, \rref t.rets/ follows from:

\begin{theorem} \label{number_of_returns_domination_thm}
$X$ stochastically dominates $X'$.
\end{theorem}

In words, larger trees have fewer returns of simple random walk for this
model.
We shall need a technical lemma for the proof.

\begin{lemma} \label{convexity_and_domination_lemma}
Fix integers $a\ge 1$ and $b\ge 0$. Let $F$ be a convex increasing function
on $\CO{0, \infty}$. Let
$X_1,\ldots,X_a,Y_1,\ldots,Y_b$ be independent non-negative random variables.
If each $X_i$ stochastically dominates each $Y_j$, then
$$
\E F\left(\frac{1}{a+b}\left(\sum_{i=1}^a X_i + \sum_{i=1}^b
Y_i\right)\right) \le \E F\left(\frac{1}{a}\left(\sum_{i=1}^a
X_i\right)\right)
\,.
$$
\end{lemma}
\begin{proof} Define an auxiliary random vector $A$ with
$a+b$ coordinates to be uniformly chosen among the $N:=\nchoosek{a+b}{a}$
vectors containing exactly $a$ values equal to $1/a$ and $b$ zeroes.
Condition on the $X$'s and $Y$'s and
consider the random variable 
$$
R:=F\left(\E_A\left(\sum_{i=1}^a A_i X_i + \sum_{i=1}^b A_{i+a}
Y_i\right)\right),
$$
where $\E_A$ denotes expectation over $A$.
On the one hand, we have
\begin{equation} \label{R_first_eval}
R = F\left(\frac{1}{a+b}\left(\sum_{i=1}^a X_i + \sum_{i=1}^b
Y_i\right)\right)
\,.
\end{equation}
On the other hand, by Jensen's inequality (since the $X$'s and $Y$'s are non-negative)
\begin{equation} \label{R_Jensen}
R \le \E_A F\left(\sum_{i=1}^a A_i X_i + \sum_{i=1}^b A_{i+a}
Y_i\right) = N^{-1}\sum_{i=1}^N M_i
\,,
\end{equation}
where each $M_i$ is a random variable of the form
$$
F\left(\frac{1}{a}\sum_{k=1}^{a_1} X_{i_k} +
\frac{1}{a}\sum_{k=1}^{b_1} Y_{j_k}\right)
\,;
$$
here, $(X_{i_k})_{k=1}^{a_1}$ is a subset of $a_1$ of the $X$'s,
$(Y_{j_k})_{k=1}^{b_1}$ is a subset of $b_1$ of the $Y$'s, $a_1\le
a, b_1\le b$ and $a_1+b_1=a$. Taking now expectation over the $X$'s
and $Y$'s and using that each $X_i$ stochastically dominates each
$Y_j$, we get
\begin{equation} \label{R_domination}
\E(M_i) \le \E F\left(\frac{1}{a}\sum_{k=1}^{a}X_k\right)
\end{equation}
since the $X_i$'s and $Y_j$'s are independent, non-negative and each $M_i$ is
increasing in each of the random variables. Putting
\eqref{R_first_eval}, \eqref{R_Jensen} and \eqref{R_domination}
together, we get
$$
\E F\left(\frac{1}{a+b}\left(\sum_{i=1}^a X_i + \sum_{i=1}^b
Y_i\right)\right) \le \E F\left(\frac{1}{a}\sum_{i=1}^{a} X_i\right)
\,,
$$
proving the lemma.
\QED
\end{proof}

\proofof{Theorem \ref{number_of_returns_domination_thm}}
It is enough to show that for each integer $M\ge 2$, we have
\begin{equation} \label{domination_ineq}
\P(X\ge M) \ge \P(X'\ge M)
\,.
\end{equation}
We couple the two trees according to the coupling given in the preceding section, in the proof of Lemma~\ref{l.les} and Theorem~\ref{t.main}, for $T_\infty(\lambda)$ and $T_\infty(\mu)$. It is then enough to show the inequality \eqref{domination_ineq} conditioned on the number of subtrees of each size that the roots of $T$ and $T'$ have (the variables $n_k(T)$ and $n_k(T')$). Henceforth we always condition on these values. 

Denote $N_\FF:=\sum_{k=1}^\infty n_k(T), N_\II:=n_\infty(T),
N_\FF':=\sum_{k=1}^\infty n_k(T')$ and $N_\II':=n_\infty(T')$. According to
the coupling, we have $N_\FF\ge N_\FF'$ and $d := N_\FF + N_\II\le N_\FF' +
N_\II' =: d'$. 
We construct our coupling of $T$ and $T'$ to have the following properties:
$T$ is a rooted subtree of $T'$; 
the children of the root are ordered;
the first $N_\FF'$ children of the root lie in $T$ and have finite
subtrees, all pairwise equal in $T$ and $T'$;
the next $N_\FF-N_\FF'$ children lie in $T$ and have finite subtrees;
and the next $N_\II$ children lie in $T$.
Given the sizes of the subtrees, recall that the pairs of coupled subtrees of
the children of the root are independent, even including the $d'-d$ left-over
subtrees of $T'$, and that all $N_\II'$ of the infinite subtrees of $T'$ are i.i.d.

For each $1\le j\le N_\FF$, suppose that
the random walk enters in its first step the $j$th finite subtree of $T$. Let $P_j^\FF$ be the probability that
continuing this random walk, we ever return to the root. Similarly,
define for $1\le j\le N_\II$ the probability $P_j^\II$ to return from
the $j$th infinite subtree of $T$. 
Define analogously $Q_j^\FF$ and $Q_j^\II$ on $T'$.
Thus, $P_j^\FF = Q_j^\FF$ for $j \le N_\FF'$ because of the coupling.

The inequality \eqref{domination_ineq} can now be written as follows:
$$
\E\left(\frac{s}{d}\left(\sum_{j=1}^{N_\FF} P_j^\FF + \sum_{j=1}^{N_\II}
P_j^\II\right)\right)^{M-1} \ge
\E\left(\frac{s}{d'}\left(\sum_{j=1}^{N_\FF'} Q_j^\FF +
\sum_{j=1}^{N_\II'} Q_j^\II\right)\right)^{M-1}
\,.
$$
We prove this inequality in two steps. First we observe that
$$
\E\left(\frac{s}{d}\left(\sum_{j=1}^{N_\FF} P_j^\FF +
\sum_{j=1}^{N_\II} P_j^\II\right)\right)^{M-1} \ge
\E\left(\frac{s}{d}\left(\sum_{j=1}^{N_\FF'} Q_j^\FF +
\sum_{j=1}^{d-N_\FF'} Q_{j}^\II\right)\right)^{M-1}
\,;
$$
this is because if the walk entered a branch of $T'$ that contains a branch
of $T$, then certainly its probability ever to return to the root
is smaller in $T'$ than it is in $T$ (by coupling the walks).
Now we may use Lemma \ref{convexity_and_domination_lemma} to get that
$$
\E\left(\frac{s}{d}\left(\sum_{j=1}^{N_\FF'} Q_j^\FF +
\sum_{j=1}^{d-N_\FF'} Q_{j}^\II\right)\right)^{M-1}
\ge
\E\left(\frac{s}{d'}\left(\sum_{j=1}^{N_\FF'} Q_j^\FF +
\sum_{j=1}^{N_\II'} Q_j^\II\right)\right)^{M-1}
$$
since each $Q_j^\II$ is stochastically dominated by each $Q_j^\FF$ (again,
by coupling the walks on the coupled subtrees). This proves the theorem.
\QED

\begin{section}{Smoothness}

Let $\bar{p}_k=\bar{p}_k(c):=\int p_k(\bp;T)d\PGW^*_c(T, \bp)$. 
We shall prove
\begin{theorem} \label{derivative_bound_thm}
For each $k \ge 1$,
$\bar{p}_k(c)$ is real analytic in $c>1$ and there exists $\beta>0$ such
that for $c> 1$, $k \ge 1$, and $n \ge 1$, we have
$$
\left|\frac{\partial^{n} \bar{p}_k}{\partial c^n}\right| \le A^n n!k^{\beta
n} e^{-ak^{1/6}},
$$
where the constants $A,a>0$ depend only on $c$ and are bounded from 0 and
infinity for $c$ in every compact subinterval of $(1,\infty)$.
\end{theorem}
\begin{remark}
We obtain $\beta=1$ in the proof, but this could be reduced further.
\end{remark}
An immediate corollary is
\begin{corollary}
$f(c)$ is $C^\infty$ for $c\in(1,\infty)$.
\end{corollary}
To prove Theorem \ref{derivative_bound_thm}, we shall prove
\begin{theorem} \label{analytic_bound_p_k}
For each $k \ge 1$,
$\bar{p}_k$ can be analytically continued to the domain $\Omega_k:=\{x+iy\
|\ x\in(1,\infty),\ |y|\le a k^{-\beta}\}$ for some $\beta>0$ and satisfies for $c\in\Omega_k$
$$
|\bar{p}_k(c)|\le Ae^{-ak^{1/6}},
$$
where $A,a>0$ depend only on $x = \Re(c)$ and are bounded from 0 and infinity for
$x$ in every compact subinterval of $(1,\infty)$.
\end{theorem}
Theorem \ref{derivative_bound_thm} is an immediate corollary by Cauchy
estimates. To see this, for each $c>1$ take a circle $C$ of radius
$r:=\min(a k^{-\beta},\frac{c-1}{2})$ around $c$. Then
$$
\left|\frac{\partial^{n} \bar{p}_k}{\partial c^n}\right| =
\left|\frac{n!}{2\pi i}\oint_{C}
\frac{\bar{p}_k(z)}{(z-c)^{n+1}}dz\right|\le
\frac{An!e^{-ak^{1/6}}}{r^n}\le AA'n!k^{\beta n}e^{-ak^{1/6}},
$$
where 
$A':=\max\left(a^{-n},\left(\frac{c-1}{2}\right)^{-n}\right)$.

In the rest of the section, we prove Theorem \ref{analytic_bound_p_k}. We
start by quoting a known result concerning a priori bounds on $\bar{p}_k$;
see \cite[Theorem 2]{\Piau98}.
\begin{theorem} \label{p_k_bound_thm}
For each $k \ge 1$, we have
$\bar{p}_k(c)\le Ae^{-ak^{1/6}}$ for $c>1$, where $A,a>0$ depend only on
$c$ and are bounded from 0 and infinity for $c$ in every compact subinterval of $(1,\infty)$.
\end{theorem}
We remark that in \cite[Theorem 2]{\Piau98} the boundedness of the constants
is not claimed, just that constants exist for every $c>1$, but this implies the theorem since $\bar{p}_k$ is a continuous function of $c$.

We now fix a compact subinterval $I\subseteq(1,\infty)$ and shall work only
with $c=x+iy$ such that $x\in I$. All the constants $A,a>0$ appearing below
may depend on $I$ and it is understood that their value may change from
line to line: $A$ may increase, while $a$ may decrease.

We record for later use the well-known structure of the $\PGW^*(c)$
distribution, as was also discussed in Section
\ref{tree_domination_section}.
\begin{lemma} \label{structure_of_CPGW_tree_lemma}
The \emph{$\PGW^*(c)$} distribution is a 2-type Galton-Watson distribution, with
the types called I and F (for ``infinite" and ``finite"). For vertices of type I the number of type I children is distributed as $Q^*_{c(1-q(c))}$ and of type F children as $Q_{cq(c)}$. Vertices of type F have only type F children, the number of which is distributed as $Q_{cq(c)}$.
\end{lemma}

We next introduce the notion of the \emph{trace of a random walk path}.
This is all the information about a path on a tree that starts at its root.
The trace includes the following information per step:
\begin{enumerate}
\item whether the step is up or down (up is away from the root);
\item if the step is up, whether it is to a type-I child or to a type-F
child and which such child is it (e.g., the first type-I child, the second
type-F child, etc.).
\end{enumerate}
We denote by $\W_k$ the set of all traces that have exactly $k$ steps and
end at the root. Given $W\in\W_k$, let $l_{\II}(W), l_{\FF}(W)$ be the number of
distinct vertices of types I and F, respectively, that the trace visits, so
that $l_{\II}(W)+l_{\FF}(W)\le k$. Given a tree $T$, the trace $W$ may be
feasible on $T$ or not: It is feasible if and only if all the vertices that
$W$ visits exist in $T$ (e.g., if on the first step, $W$ moves to the third
type-F child of the root, then the root of $T$ must have at least 3
type-F children). Let $d_{\II}(W):=(d_1^\FF,d_1^\II,\ldots,
d_{l_{\II}(W)}^\FF,d_{l_{\II}(W)}^\II)$ be the minimum required number of children of type F and type I from each
of the vertices of type I that $W$ visits in order for the walk to be
feasible. Here,
the subscript $i$ indicates the $i$th distinct vertex of type I visited by $W$. Similarly, let $d_{\FF}(W):=(\tilde{d}_1^{\FF},\ldots, \tilde{d}_{l_{\FF}(W)}^{\FF})$ be the minimum required number of children of type F from the vertices of type F that $W$ visits. 
Given two vectors $e$ and $d$ of the same length,
we write $e\succeq d$ if each coordinate of $e$ is
greater than or equal to the corresponding coordinate of $d$.
Finally, denote by $p^c(W)$ the probability under $\PGW^*(c)$ to sample a
feasible tree for $W$ and then to sample $W$ as a simple random walk path
of length $k$ on that tree. From all the above discussion, we have
\begin{equation} \label{main_probability_sum}
\bar{p}_k(c)=\sum_{W\in\W_k} p^c(W) = \sum_{W\in\W_k}\sum_{\substack{e_{\II}\succeq d_{\II}(W)\\e_{\FF}\succeq d_{\FF}(W)}}
p^c(e_{\II},e_{\FF})p(W,e_{\II},e_{\FF}),
\end{equation}
where $p^c(e_{\II},e_{\FF})$ is the probability to sample a tree in which the vertices
that $W$ passes through have exactly the prescribed number of children $e_{\II},e_{\FF}$
of each type, and $p(W,e_{\II},e_{\FF})$ is the conditional probability,
given $e_{\II}$ and $e_{\FF}$, to sample $W$ as a
simple random walk path on the tree. We emphasize that
$p(W,e_{\II},e_{\FF})$ does not depend on $c$, while $p^c(e_{\II},e_{\FF})$
is the same for all $W$ that satisfy $e_{\II}\succeq
d_{\II}(W)$ and $e_{\FF}\succeq d_{\FF}(W)$.

Since $c$ and $q$ are analytically related for $c\in(1,\infty)$ by \rref
e.extinct/ (which can be rewritten as $c=-\frac{\log(q)}{1-q}$),
there is a unique extension of $q(c)$ to an analytic function of $c$
for $\Re c > 1$ and $|\Im c| \le \kappa(\Re c)$ for some continuous
function $\kappa : (1, \infty) \to (0, \infty)$. 
(In fact, one can extend it much further, but we shall not need that.)
Hence, the same holds for $p^c(e_\II, e_\FF)$.
We shall use the same notations for the original functions as
for these extensions, and likewise for similar functions below.

Note that to prove Theorem \ref{analytic_bound_p_k}, it is enough to show
that for $c=x+iy$ with $x\in I$ and $|y|\le \kappa(x) k^{-\beta}$, the sum
\eqref{main_probability_sum} converges uniformly and is bounded by
$A\exp(-ak^{1/6})$. Denote by $\max(e)$ the maximal element of $e$. We continue with
\begin{lemma}\label{complex_to_real_lemma}
If $c=x+iy$ with $x\in I$ and $|y| \le \kappa(x)$, then
$$
|p^c(e_{\II},e_{\FF})|\le p^x(e_{\II},e_{\FF})e^{Ak(\max(e_{\II})+\max(e_{\FF})+1)|y|}.
$$
\end{lemma}
\begin{proof}
Using the structure Lemma \ref{structure_of_CPGW_tree_lemma}, we know that
\begin{equation*}
p^c(e_{\II},e_{\FF})=\prod_{i=1}^{l_{\II}(W)}p^c(e_{\II},i)\prod_{i=1}^{l_{\FF}(W)}p^c(e_{\FF},i)
\end{equation*}
(the values $l_{\II}(W)$ and $l_{\FF}(W)$ are implicit in $e_{\II}$ and
$e_{\FF}$ as their lengths), where
\begin{equation*}
\begin{split}
p^c(e_{\II},i) &= \P(Q^*_{c(1-q(c))}=e_i^\II)\P(Q_{cq(c)}=e_i^\FF)\,,\\
p^c(e_{\FF},i) &= \P(Q_{cq(c)}=\tilde{e}_i^\FF)
\,.
\end{split}
\end{equation*}
More explicitly, denoting $j:=e_i^\II, m:=e_i^\FF$ and $n:=\tilde{e}_i^{\FF}$ and abbreviating
$q:=q(c)$, we have
\begin{equation*}
\begin{split}
p^c(e_{\II},i) &=
e^{-c(1-q)}\frac{(c(1-q))^j}{j!(1-e^{-c(1-q)})}e^{-cq}\frac{(c q)^m}{m!}\,,\\
p^c(e_{\FF},i) &= e^{-c q}\frac{(c q)^n}{n!}\,.
\end{split}
\end{equation*}
We have in the first case
\begin{equation}\label{vertex_bound}
|p^c(e_{\II},i)|\le
\frac{e^{-x}}{j!m!}|c|^{j+m}|1-q|^j|q|^m\frac{1}{|1-e^{-c(1-q)}|}\,.
\end{equation}
We know that when $x\in I$ and $|y| \le \kappa(x)$, we have $|q(x+iy) - q(x)|\le A|y|$ since $q$ is an analytic function of $c$. Hence
$$
|c|^{j+m} = x^{j+m}\left|1+\frac{y^2}{x^2}
\right|^{\frac{j+m}{2}}\le x^{j+m}e^{Ay^2(j+m)}\,,
$$
$$
|1-q(x+iy)|^j \le (1-q(x))^j \left(1+\frac{A|y|}{1-q(x)}\right)^j\le
(1-q(x))^je^{Aj|y|}\,,
$$
$$
|q(x+iy)|^m \le q(x)^m \left(1+\frac{A|y|}{q(x)}\right)^m\le
q(x)^me^{Am|y|}\,,
$$
and
$$
\begin{aligned}
\frac{1}{|1-e^{-(x+iy)(1-q(x+iy))}|}&\le \frac{1}{1-e^{-x(1-q(x))}} + A|y| \le \frac{1}{1-e^{-x(1-q(x))}}(1+A|y|)\\
&\le \frac{e^{A|y|}}{1-e^{-x(1-q(x))}}.
\end{aligned}
$$
Substituting back into \eqref{vertex_bound}, we get
$$
|p^{c}(e_{\II},i)|\le p^{x}(e_{\II},i)e^{A(j+m+1)|y|}\,.
$$
This bound was for $p^c(e_{\II},i)$, but we also obtain analogously that $|p^{c}(e_{\FF},i)|\le p^{x}(e_{\FF},i)e^{A(n+1)|y|}$. Hence
$$
|p^c(e_{\II}, e_{\FF})|\le p^x(e_{\II},e_{\FF})e^{A(l_{\II}(W)+l_{\FF}(W))(\max(e_{\II}) + \max(e_{\FF}) +1)|y|}\le p^x(e_{\II},e_{\FF})e^{Ak(\max(e_{\II})+\max(e_{\FF})+1)|y|}
.
\QED
$$
\end{proof}

To continue, say that a vertex of a tree is $L$-big if it has either exactly
$L$ type-I children or exactly $L$ type-F children or both. Let $E_{k,L}$
be the event that if we sample a tree and do a simple random walk on it
(from the root), then the walk returns to the root after exactly $k$ steps
and visits an $L$-big vertex along the way but does not visit an $M$-big
vertex along the way for any $M>L$. We observe that
\begin{equation}\label{large_L_bound}
\sum_{W\in\W_k}\sum_{\substack{e_{\II}\succeq d_{\II}(W)\\e_{\FF}\succeq d_{\FF}(W)\\\max(\max(e_{\II}),\max(e_{\FF}))=L}} p^x(e_{\II},e_{\FF})p(W,e_{\II},e_{\FF}) =
\PGW^*_x(E_{k,L})\le Ake^{-aL\log L},
\end{equation}
where the last inequality follows since there are no more than $k$ vertices
along any path $W$ and since the tails of a Poisson$(c)$ random variable decay as $Ae^{-aL\log L}$, even when conditioned to be at least 1.

Thus, we find that if $c=x+iy$ with $x\in I$ and $|y|\le \kappa(x)/k$, then
from Lemma \ref{complex_to_real_lemma} (for $k$ large enough as a function
of $I$), we have
$$
\begin{aligned}
|\bar{p}_k(c)| &\le \sum_{W\in\W_k}\sum_{\substack{e_{\II}\succeq d_{\II}(W)\\e_{\FF}\succeq d_{\FF}(W)}} |p^c(e_{\II},e_{\FF})|p(W,e_{\II},e_{\FF})\\
&\le \sum_{W\in\W_k}\sum_{L=1}^\infty e^{A(L+1)}\sum_{\substack{e_{\II}\succeq d_{\II}(W)\\e_{\FF}\succeq d_{\FF}(W)\\\max(\max(e_{\II}),\max(e_{\FF})))=L}}p^x(e_{\II},e_{\FF})p(W,e_{\II},e_{\FF})\\
&\le\underbrace{e^{A\delta k^{1/6}}\sum_{W\in\W_k}\sum_{L\le \delta
k^{1/6}}(\cdots)}_{(C)}+\underbrace{\sum_{W\in\W_k}\sum_{L>\delta
k^{1/6}}e^{AL}(\cdots)}_{(D)}\,.
\end{aligned}
$$
By \eqref{large_L_bound}, we have
$$
(D)\le Ak\sum_{L>\delta k^{1/6}} e^{AL-aL\log L} \le Ae^{-a\delta
k^{1/6}\log k}\,,
$$
and by Theorem \ref{p_k_bound_thm}, we have
$$
(C)\le e^{A\delta k^{1/6}}\bar{p}_k(x)\le Ae^{(A\delta-a)k^{1/6}} \le
Ae^{-ak^{1/6}}\,,
$$
where the last inequality follows by taking $\delta$ small enough (as a
function of $I$). Putting everything together, we get
$$
|\bar{p}_k(c)|\le Ae^{-ak^{1/6}}.
$$
The calculation was made for $k$ large enough as a function of $I$, but the inequality will be true for smaller $k$ as well by taking $A$ large enough.
This completes the proof of Theorem \ref{analytic_bound_p_k}.
\end{section}

\section{Derivative} \label{deriv}

By \rref t.return/ and \rref e.treeent/, we have
$$
f'(c) \ge
\frac{d}{d c}
\int \log \deg_T(\bp) \,d\PGW^*_c(T, \bp)
=
\frac{d}{d c}
\sum_{k \ge 1} \frac{e^{-c} c^k (1 - q(c)^k) \log k}{\theta(c) k!}
\,.
$$
Although this lower bound appears to be a fairly simple expression, the
presence of the logarithm makes it hard to evaluate.
For that reason, it seems desirable to have a more explicit lower bound.

Write $r_k(c)$ for the probability that the root has degree $k$ under
the $\PGW^*(c)$ distribution.
We seek a lower bound for
$$
\begin{aligned}
\sum_{k \ge 1} r'_k(c) \log k
&=
\sum_{k \ge 0} r'_k(c) \log^+ k
=
\sum_{k \ge 0} s'_k(c) [\log^+ (k+1) - \log^+ k]
\\ &
=
\sum_{k \ge 1} s'_k(c) \log \frac{k+1}{k}
>
\sum_{k \ge 1} s'_k(c) \frac{1}{k+1}
\,,
\end{aligned}
$$
where $s_k(c) := \sum_{j > k} r_j(c)$ and we have used Lemma~\ref{l.les} for the
fact that $s'_k(c) \ge 0$ (i.e., the degree distribution of the root under
$\PGW^*(c)$ is stochastically increasing in $c$).
Now the degree of the root has the same law as $X_c := Q^*_{c \theta(c)} +
Q_{c q(c)}$.
Let $c > 1$ and $\delta > 0$.
Define
$$
g(c, \delta)
:=
\alpha\big(c \theta(c), (c+\delta)\theta(c+\delta)\big) -
[c q(c) - (c+\delta)q(c+\delta)]
\,.
$$
By \rref l.ders/, we have $g(c, \delta) > 0$.
By Lemma~\ref{l.decomp}, we have that 
$X_{c+\delta}$ stochastically dominates 
$$
Q^*_{c \theta(c)} +
Q_{\alpha(c \theta(c), (c+\delta)\theta(c+\delta))} +
Q_{(c+\delta)q(c+\delta)}
\,,
$$
which has the same distribution as $X_c + Y_{c, \delta}$, 
where $Y_{c, \delta} := Q_{g(c, \delta)}$ is independent of $X_c$.
Therefore,
$$
s_k(c+\delta) - s_k(c)
=
\Pb{X_{c+\delta} > k} - \Pb{X_c > k}
\ge
\Pb{X_c + Y_{c, \delta} > k} - \Pb{X_c > k}
\,.
$$
It follows that
$$
s'_k(c)
\ge
r_k(c) \beta(c)
\,,
$$
where
$$
\beta(c)
:=
\lim_{\delta \to 0} g(c, \delta)/\delta
\,.
$$
By \rref e.simpler/ and \rref e.duality/, we have that 
$$
\begin{aligned}
g(c, \delta)
&=
\log \big(c q(c) e^{-c q(c)}\big) - \log \big((c+\delta) q(c+\delta)
e^{-(c+\delta) q(c+\delta)}\big)
\\ &=
\log \big(c e^{-c}\big) - \log \big((c+\delta) e^{-(c+\delta)}\big)
=
\log c - c - \log (c+\delta) + c+\delta
\,.
\end{aligned}
$$
Therefore,
$$
\beta(c)
=
1 - \frac{1}{c}
\,.
$$
Thus, we obtain
$$
\begin{aligned}
f'(c)
&>
\sum_{k \ge 1} \frac{r_k(c) \beta(c)}{k+1}
=
\frac{e^{-c} \beta(c)}{\theta(c)}
\int_0^1 \Big(e^{c s} - e^{c s q(c)}\Big) \,d s
\\ &=
\left(
1 - \frac{1}{c}
\right)
\left(\frac{1 - e^{-c}}{c \theta(c)} - \frac{e^{-c \theta(c)} - e^{-c}}{c
q(c) \theta(c)}\right)
\\ &=
\frac{(c-1)e^{-c q(c) }}{c^2}
> 0
\,.
\end{aligned}
$$
This completes the proof of \rref t.incrH/.

\section{Open problems}

A number of questions suggest themselves in light of our results, some of
which arose in conversation with Itai Benjamini. 
\begin{enumerate}
\item 
Given two finite graphs $H$ and $G$,
say that $H \preccurlyeq G$ if there is a coupling of uniform
vertices $Y$ of $H$ and $X$ of $G$ such that there is an isomorphism
$\varphi$ of the
component of $Y$ in $H$ to a subgraph of $G$ such that $\varphi(Y) = X$. 
Let $\G(n, M)$ denote the random graph on $n$ vertices with $M$ edges.
Write $\G^*(n, M)$ for the union of all components of $\G(n, M)$ that have
the maximum number of edges (the maximum being taken over all the
components of $\G(n, M)$; for large $M$, there is likely to be only one
such component).
One very strong finitary version of Theorem \ref{t.main} would say that
$\G^*(n, M) \preccurlyeq \G^*(n, M+1)$ for $M < \binom{n}{2}$.
Does this hold?
\item
Consider a $(d+1)$-regular tree and $p_2 > p_1 > 1/d$. Let $T(p)$ denote the
component of the root under Bernoulli($p$) percolation conditioned on the
event that this component is infinite. Does $T(p_2)$ stochastically
dominate $T(p_1)$? 
Unpublished work of Erik Broman and the first author here shows that
for the (slightly different) case of $d$-ary trees, this holds for $d = 2, 3$.
\item
More generally, let $G$ be a transitive graph, especially such as $\Z^d$,
and $p_2 > p_1 > \pc(G)$, where $\pc(G)$ is the critical probability for
Bernoulli (bond or site) percolation on $G$. Fix $o \in G$ and
let $G(p)$ denote the component of $o$
given that it is infinite. Does $G(p_2)$ stochastically dominate $G(p_1)$?
If this holds, then there is a weak limit of $G(p)$ as $p \downarrow
\pc(G)$, which could be called the incipient infinite cluster. (It is
conjectured that there is no infinite component at $\pc(G)$; see
\cite{\BSpyond}.)
Such a limit is not known to exist in $\Z^d$ for $d \ge 3$, although
another incipient infinite cluster has been constructed for $d \ge 19$ by
\cite{\HoJa}. 
\item
Again, if $G$ is a transitive graph, $o \in G$,
and $n \ge 1$, let $T_n$ denote a
uniformly chosen random subtree of $G$ rooted at $o$ and with $n$ vertices.
Is $T_n \preccurlyeq T_{n+1}$? 
\item
Let $\rtd_1$ and $\rtd_2$ be two Galton-Watson measures on rooted trees.
If $k \ge 2$ and $\rtd_1 \preccurlyeq \rtd_2$, then is it necessarily the
case that $\int p_k(\bp; T) \,d\rtd_1(T, \bp) \ge
\int p_k(\bp; T) \,d\rtd_2(T, \bp)$?
\end{enumerate}
We thank Yuval Peres for several conversations.

\bibliographystyle{halpha}
\bibliography{mr,prep}

\end{document}